%% file: paradox_nonlinear.tex
\documentclass[11pt]{amsart}
\usepackage{stmaryrd}
\usepackage{amssymb}
\usepackage{csvsimple}
\usepackage{siunitx}
\usepackage{hyperref}
\usepackage{url}
\usepackage{graphicx,color}
\usepackage{lineno}           
\usepackage[style=trad-abbrv,url=false,doi=true,
  eprint=true,isbn=false,backend=biber]{biblatex}

\graphicspath{{pics/},{figs/}}
\include{sb_macros}

\renewcommand{\text}{\textnormal}
\newcommand{\rx}[1]{{\color{black} #1}}

\def\ho{{\widehat{\o}}}
\def\hk{{\widehat{\k}}}
\def\tV{\widetilde{V}}
\def\tI{\widetilde{I}}
\def\tv{\widetilde{v}}
\def\hg{{\widehat{\g}}}

\def\cmu{\widehat{\mu}}

\def\tcP{\widetilde{\cP}}

\usepackage{pgfplots, tikz}
\pgfplotsset{compat=1.15}

\usetikzlibrary{calc}
\def\centerarc[#1](#2)(#3:#4:#5)
{ \draw[#1] ($(#2)+({#5*cos(#3)},{#5*sin(#3)})$) arc (#3:#4:#5); }

\newcommand{\bM}{\boldsymbol{m}}

\renewcommand{\L}{\mathcal{L}}
\newcommand{\sym}{\mathrm{sym}}

\newcommand{\E}{\mathcal{E}}

\newcommand{\Hessian}{\mathcal{H}}

\begin{document}
\title{Babu\v{s}ka's paradox in a nonlinear \rx{bending-folding} model}
\author[S. Bartels]{S\"oren Bartels}
\address{Abteilung f\"ur Angewandte Mathematik,
Albert-Ludwigs-Universit\"at Freiburg, Hermann-Herder-Str.~10,
79104 Freiburg i.~Br., Germany}
\email{bartels@mathematik.uni-freiburg.de}
\author[A. Bonito]{Andrea Bonito}
\address{Texas A\& M University, College Station, TX 77843, USA}
\email{bonito@tamu.edu}
\author[P. Hornung]{Peter Hornung}
\address{Fakult\"at Mathematik, Technische Universit\"at Dresden,
Zellescher Weg 12--14, 01069 Dresden, Germany}
\email{peter.hornung@tu-dresden.de}
\author[M. Neunteufel]{Michael Neunteufel}
\address{Department of Mathematics and Statistics, Portland State University, Portland OR 97201, USA}
\email{mneunteu@pdx.edu}
\date{\today}
\renewcommand{\subjclassname}{
\textup{2010} Mathematics Subject Classification}
\begin{abstract}
The Babu\v{s}ka or plate paradox concerns the failure of convergence
when a domain with curved boundary is approximated by polygonal domains
in linear bending problems with \rx{simply supported boundary} conditions. It can
be explained via a boundary integral representation of the total 
Gaussian curvature that is part of the Kirchhoff--Love bending energy. 
It is shown that the paradox also occurs for a nonlinear bending-folding 
model which enforces vanishing Gaussian curvature. A simple remedy that is
compatible with simplicial finite element methods to avoid 
\rx{incorrect} convergence is devised. 
\end{abstract}

\maketitle

\section{Introduction} 

\subsection{Babu\v{s}ka's paradox}
A remarkable observation due to Babu\v{s}ka, cf.~\cite{BabPit90}, is that canonical approximations
of certain fourth order problems may fail to converge when curved domains are
approximated using polygons. \rx{In particular, this occurs for the Kirchhoff--Love 
bending energy 
\[
I(\o;v) = \frac{\s}{2} \int_\o |\Delta v |^2\dv{x} + \frac{1-\s}{2} \int_\o |D^2 v|^2 \dv{x}
\]
defined on the set of functions $v\in V(\o) = H^2(\o)\cap H^1_0(\o)$ corresponding to simply
supported boundary conditions. Then, the approximating functionals $I(\o_m;\cdot)$
with a sequence of approximating polygons $\o_m$ for $\o$ do not converge 
in variational sense to $I(\o,\cdot)$ if $\o$ has curved boundary parts.} The incorrect
convergence is illustrated in Figure~\ref{fig:paradox_linear}.
A simple explanation follows from the relation 
\[\begin{split}
\int_\ho |D^2 v|^2 \dv{x} &= \int_\ho |\Delta v|^2 - 2 \det D^2 v \dv{x} \\
&=  \int_\ho  |\Delta v|^2 \dv{x} - \int_{\p\ho} \hk |\p_\nu v|^2 \dv{s},
\end{split}\]
which holds for Lipschitz domains $\ho\subset \R^2$ whose boundaries consist 
of finitely many $C^2$ arcs whose piecewise curvature is denoted by $\hk$ which
is positive for locally convex arcs, cf.~\cite{CoNiSw19} for a related density result.
\rx{To address the variational convergence of functionals associated with approximating
domains $\o_m\subset \o$ we trivially extend functions and their derivatives defined on $\o_m$ to
functions on $\o$ by assigning the value zero in $\o\setminus \o_m$.}
Since the functionals $I(\o_m;\cdot)$ and $I(\o;\cdot)$ are quadratic expressions in 
\rx{the Hessian} we have
that if $(v_m)_{m\ge 0} \subset H^1_0(\o)$ is a sequence with 
$v_m\in H^2(\o_m)\cap H^1_0(\o_m)$ and $v_m \wto v$ in $H^1_0(\o)$ then
$I(\o_m;v_m) \to I(\o;v)$ implies that the trivial extensions of $D^2 v_m$ converge
strongly to $D^2 v$ in $L^2(\o)$. This leads to a contradiction since the
boundary integral terms in $I(\o_m;\cdot)$ vanish for every $m\ge 0$ but provide a
nontrivial contribution to $I(\o;\cdot)$ unless the boundary of $\o$ is piecewise
straight or the normal derivative vanishes on $\p \o$. Note that also distributional curvature 
contributions in the corner points do not improve the convergence since in
those points we have $\p_\nu v=0$. The paradox is thus a consequence of an insufficient
convergence of the approximating boundary curvatures~$\k_m$.

\begin{figure}[htb]
\includegraphics[width=3cm]{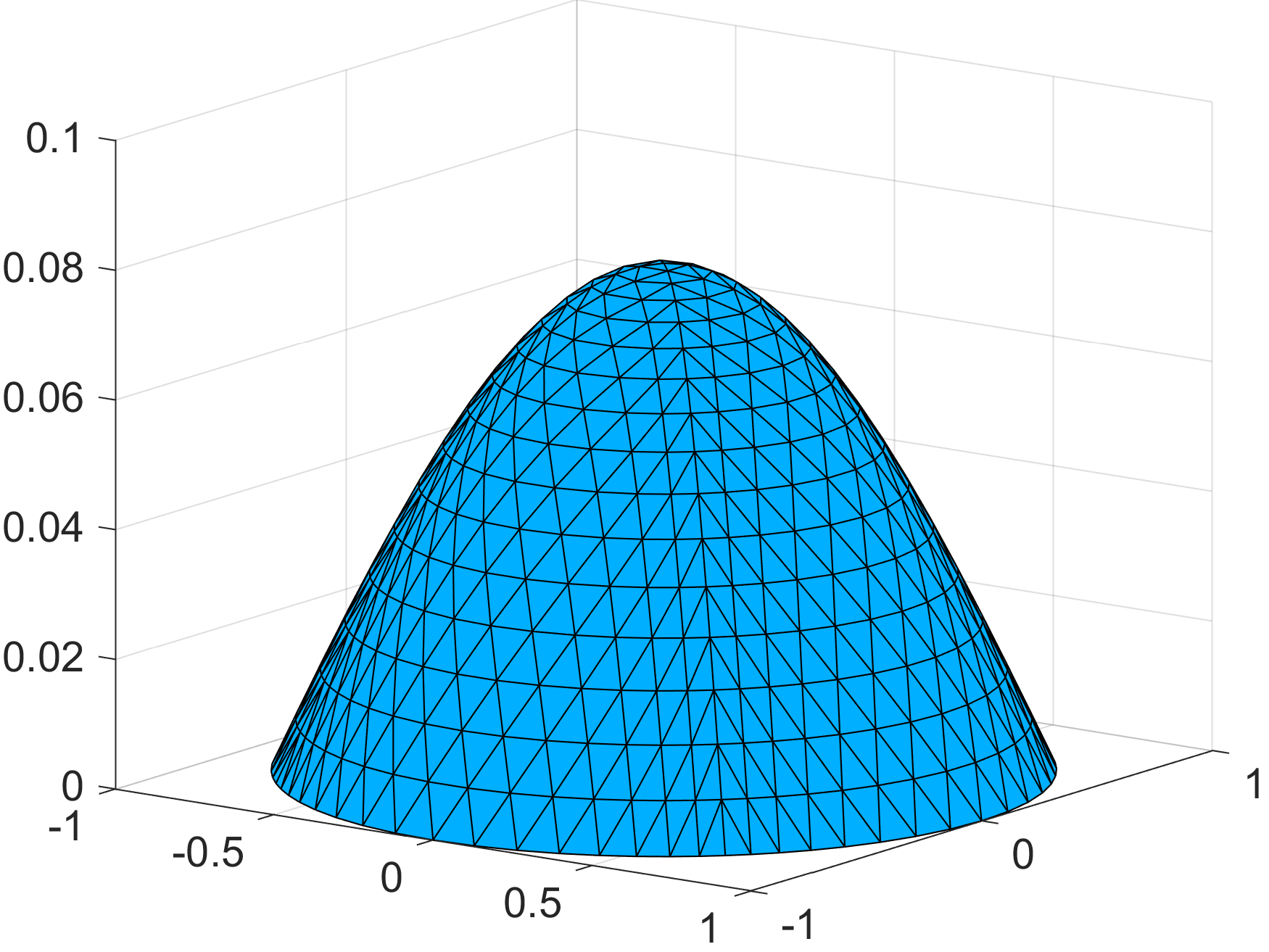} \hspace*{3mm}
\includegraphics[width=3cm]{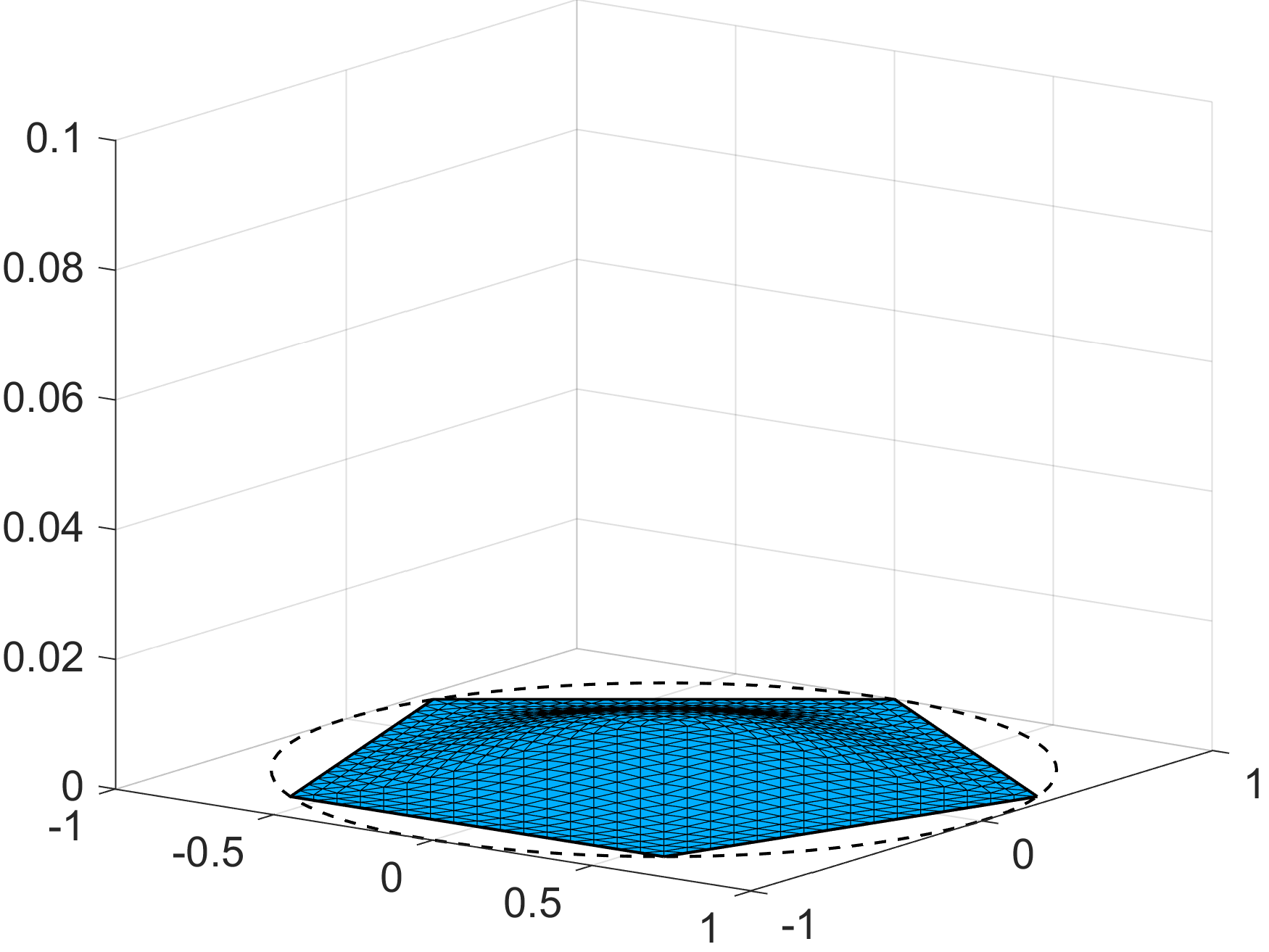} \hspace*{3mm}
\includegraphics[width=3cm]{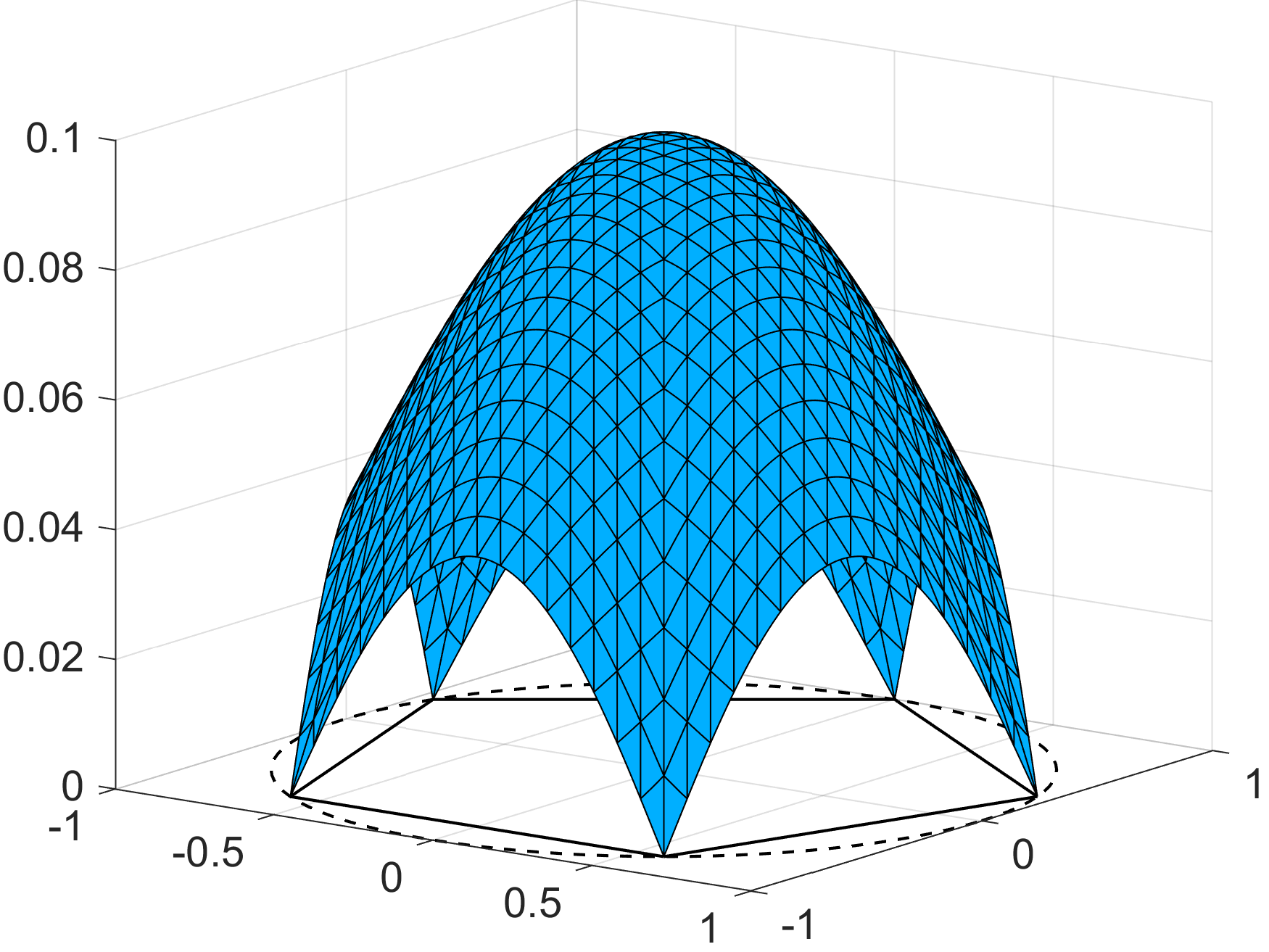} 
\caption{\label{fig:paradox_linear} Interpolant of the exact deflection (left),
approximation imposing the boundary condition along the entire boundary (middle),
and approximation obtaind imposing the boundary condition in the corner points (right);
pictures taken from~\cite{BarTsc24}.}
\end{figure}

\subsection{Remedies}
If domains $\o_m$ with piecewise quadratic boundaries are used to approximate $\o$
then the boundary curvatures converge as functions and the paradox is avoided.
To avoid the failure of convergence and still permit the use of polygonal approximations 
it suffices to relax the boundary condition in the approximating problems
by imposing it only in the corner points of $\o_m$ and in parts where $\p\o$
and $\p\o_m$ coincide, cf.~\cite{Rann79,BarTsc24}. Figure~\ref{fig:paradox_linear} illustrates the modification.
Hence, we set
\[
\tV_m = \{ v\in H^2(\o_m): v= 0 \text{ on } \p\o\cap\p\o_m\}.
\]
In this case, it is straightforward to establish a $\G$ convergence result \rx{with 
respect to weak convergence in $H^1_0(\o)$} for
$\tI(\o_m;\cdot)$ with admissible sets $\tV_m$ to $I(\o;\cdot)$ defined on $V(\o)$.
If $v_m \to v$ in $L^2$ with functions $v_m \in \tV_m$
and $I(\o_m;v_m)$ is bounded then
the trivial extensions of the Hessians $D^2 v_m$ converge weakly in $L^2(\o)$
to $D^2 v$ while suitable interpolants $\cI_m v_m \in H^1_0(\o_m)$ converge
weakly to $v\in H^1_0(\o)$. Hence, $v\in V$ and 
$I(\o;v) \le \liminf_{m\to \infty} I(\o_m;v_m)$. Given a function $v\in V$, the
restrictions $v_m = v|_{\o_m}$ belong by construction to the spaces $\tV_m$ and we have
that $I(\o_m;v_m) \to I(\o;v)$ as $m\to \infty$. We refer to~\cite{Davi02,NaSwSt11,ArnWal20} 
for other approximations that lead to correct convergence. 

\subsection{Bending and folding} 
Nonlinear bending models often involve an isometry condition on deformations
which implies that the Frobenius norms of the Hessian and the Laplacian coincide.
Recalling that their difference causes the critical boundary term in the small 
deflection model raises the question whether the paradox may in fact be an
artifact of an oversimplistic model reduction. \rx{For deformations instead of
deflections the Hessians and the Laplacians contain more information
and the above argument does not apply.} 
To show that failure of convergence also takes place in models describing 
large bending deformations we consider a Kirchhoff model that allows
for the folding of thin elastic sheets along a given crease 
line $\g \subset \overline{\o}$. This has applications in the construction
of biomimetic devices, cf.~\cite{Saff_etal_19}; cf.~\cite{ConMag08} for related 
mathematical probems. 
Along the crease line only continuity is imposed while away from it nonlinear 
bending is measured via a piecewise Kirchhoff bending energy, i.e.,
following~\cite{FrJaMu02,BaBoHo22,LiuJam24} we consider the minimization of the functional
\[
I(v) = \frac12 \int_{\o\setminus \g} |D^2 v |^2 \dv{x}.
\]
Shearing and stretching effects are inadmissible which is encoded by
the condition that deformations are isometries. Incorporating also the continuity
across the crease line, the set of admissible deformations is given by
\[
V(\o,\g) = \big\{v\in H^2(\o\setminus \g;\R^3)\cap W^{1,\infty}(\o;\R^3): 
(\nabla v)^\transp (\nabla v) = I_{2\times 2}\big\}.
\]
For finite element discretizations it is attractive to approximate 
$\g$ by a sequence of piecewise straight curves $\g_m$. This gives rise
to the approximating energy functionals
\[
I_m(v) =  \frac12 \int_{\o\setminus \g_m} |D^2 v |^2 \dv{x}
\]
in the admissible sets
\[
V(\o,\g_m) = \big\{v\in H^2(\o\setminus \g_m;\R^3)\cap W^{1,\infty}(\o;\R^3): 
(\nabla v)^\transp (\nabla v) = I_{2\times 2}\big\}.
\]
It turns out however, that the folding of isometries is impossible along
polygons as this leads to singularities in the deformation gradients
which prevent a correct convergence. 

\subsection{Approximations using slits}
As in the linear setting, we relax the continuity condition by using a perforation
and imposing continuity only at
the vertices  $c_0,c_1,\dots,c_m$ of the curves $\g_m$, i.e., \rx{considering
admissible sets}
\[\begin{split}
\tV(\o,\g_m) &= \big\{v\in H^2(\o\setminus \g_m;\R^3)\cap W^{1,\infty}(\o\setminus \g_m;\R^3): 
(\nabla v)^\transp (\nabla v) = I_{2\times 2}, \\
&\qquad  v \text{ continuous in } c_0,c_1,\dots,c_m \big\}.
\end{split}\]
To establish the variational convergence of \rx{approximations $\tI_m$ with admissible sets $\tV_m$
to $I$ defined on $V$ we define fattened crease lines $\hg_m$ as unions of 
triangles along the polygonal crease lines $\g_m$,} \rx{so that the exact crease
line is contained in the union of these triangles. Consequently,} the resulting
approximating subdomains $\ho_m^i$, $i=1,2$, are contained in the exact subdomains 
$\o^i$, $i=1,2$, cf.~Figure~\ref{fig:poly_crease}. \rx{These inclusion consistencies 
allow for the straightforward construction of a recovery sequence but do not appear
to matter in numerical simulations.}
To identify a limit $v$ of a sequence $(v_m)$ of deformations $v_m \in \tV_m$,
we choose triangulations $\cT_m$ that contain the triangles that define the
fattened crease line $\hg_m$ and carry out a linear nodal interpolation of $v_m$ which defines
functions $\cI_m v_m \in W^{1,\infty}(\o)$ with interpolation error
$\cI_m v_m - v_m$ that converges strongly to zero in $H^1(\o';\R^3)$ for domains 
$\o'$ with positve distance to the exact crease line $\g$.

\begin{figure}[htb]
\scalebox{.9}{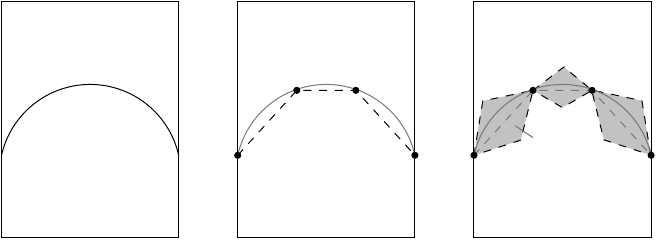}
\caption{\label{fig:poly_crease} Crease line $\g$ (left), polygonal approximation $\g_m$
(middle), and fattened polygonal crease line $\hg_m$ (right). In each case the lines
define partitions of the domain $\o$. For the fattened crease line $\hg_m$ we have
$\ho^i_m \subset \o^i$, $i=1,2$.}
\end{figure}

\rx{We remark that our arguments to establish the convergence of approximations 
also apply when a curved boundary part instead of a crease line is approximated 
by a polygon to avoid possible incorrect convergence.}

\subsection{Confirmation by experiments}
The locking effect introduced by approximating curved crease lines by
polygonal ones and still imposing continuity is confirmed by real and numerical
experiments. Figure~\ref{fig:paradox_fold_real} shows an experiment with a thin
elastic sheet and curved crease line. When it is approximated using a simple polygonal
curve and continuity is imposed along the entire line, stress concentrations occur 
at the vertices which are accomponied by piecewise flat deformations in a neighborhood
of the crease line. Introducing slits along the segments relaxes the situation and the 
experiment indicates correct convergence. 

\begin{figure}[htb]
\includegraphics[width=3.8cm]{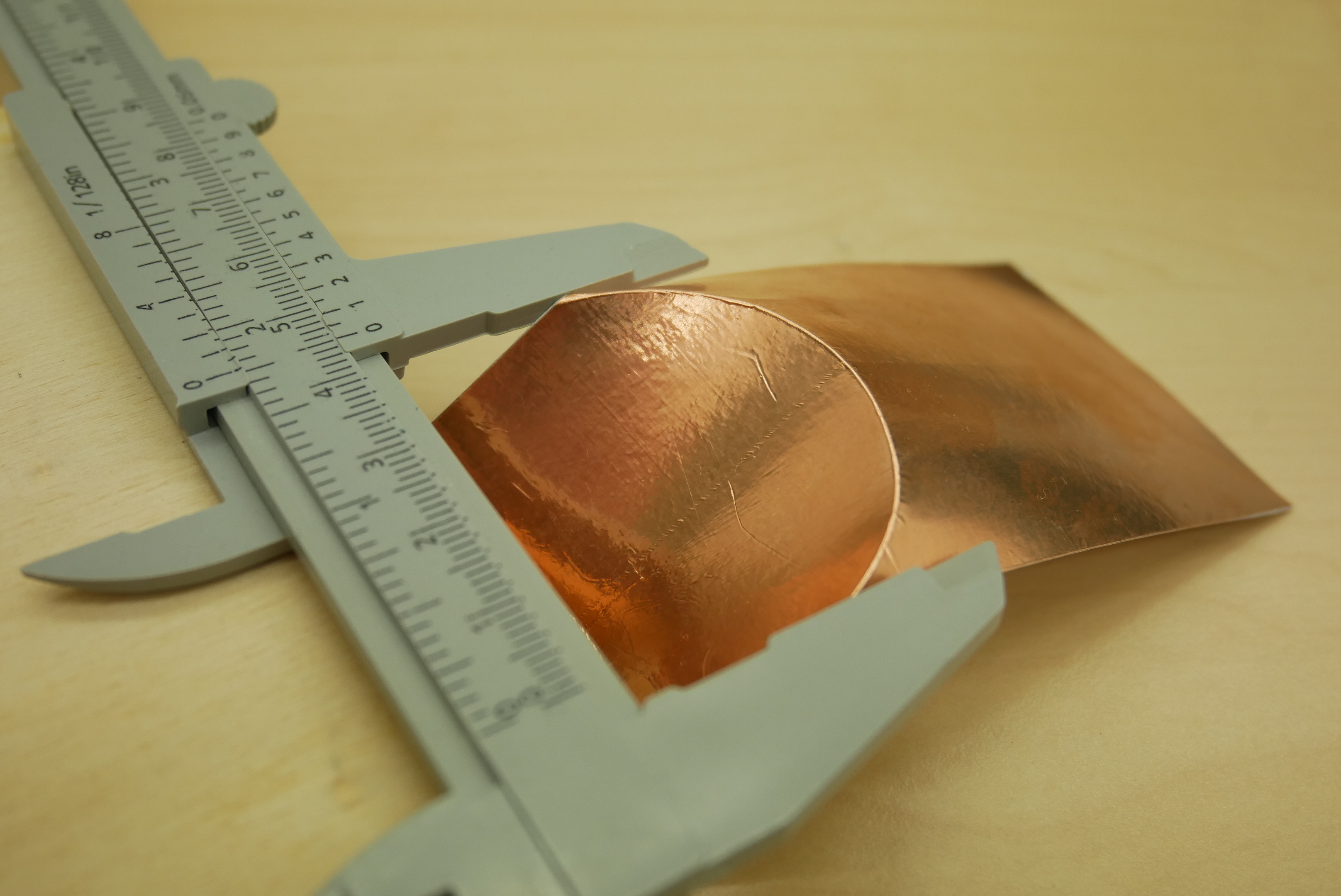} \hspace*{3mm}
\includegraphics[width=3.8cm]{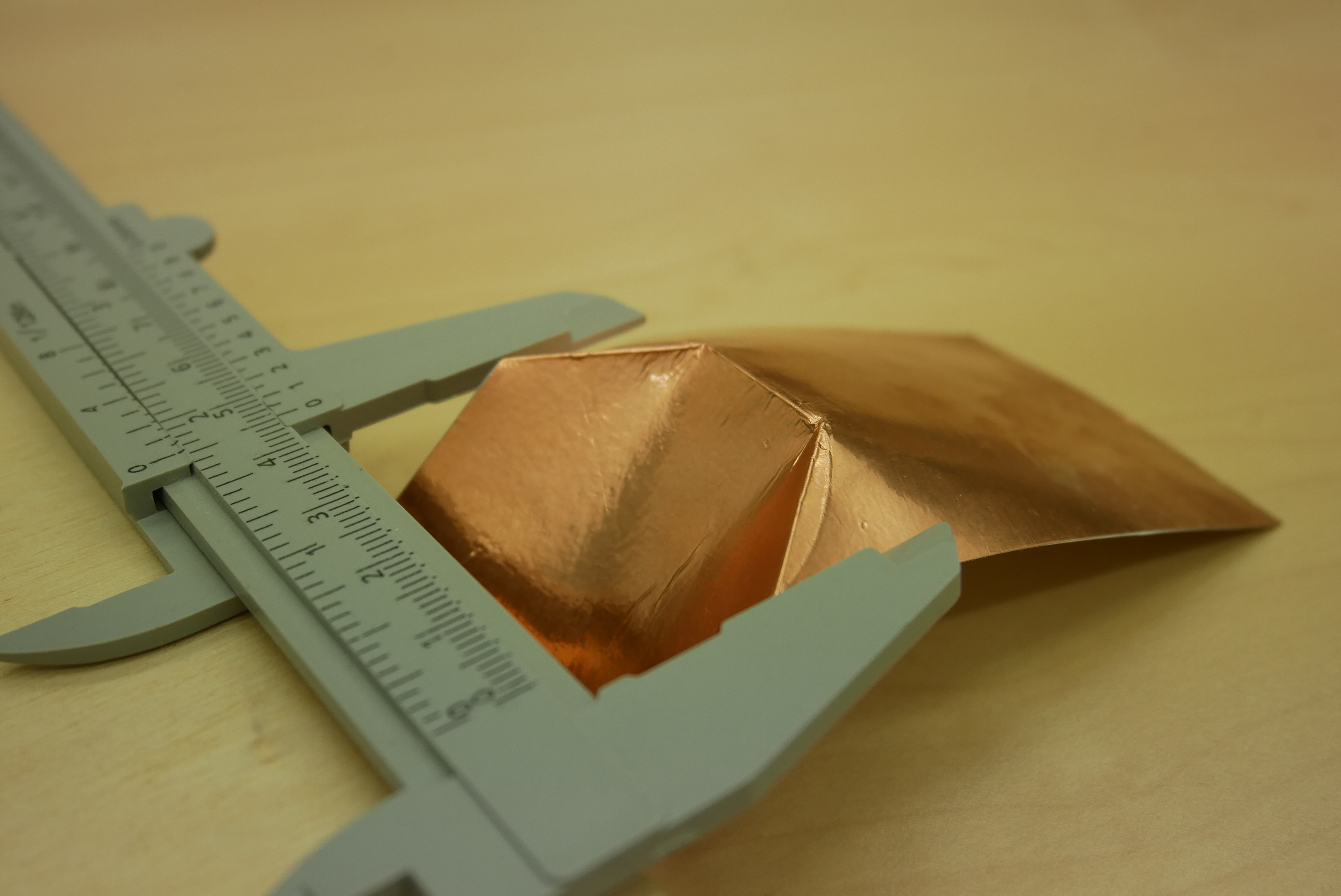} \hspace*{3mm}
\includegraphics[width=3.8cm]{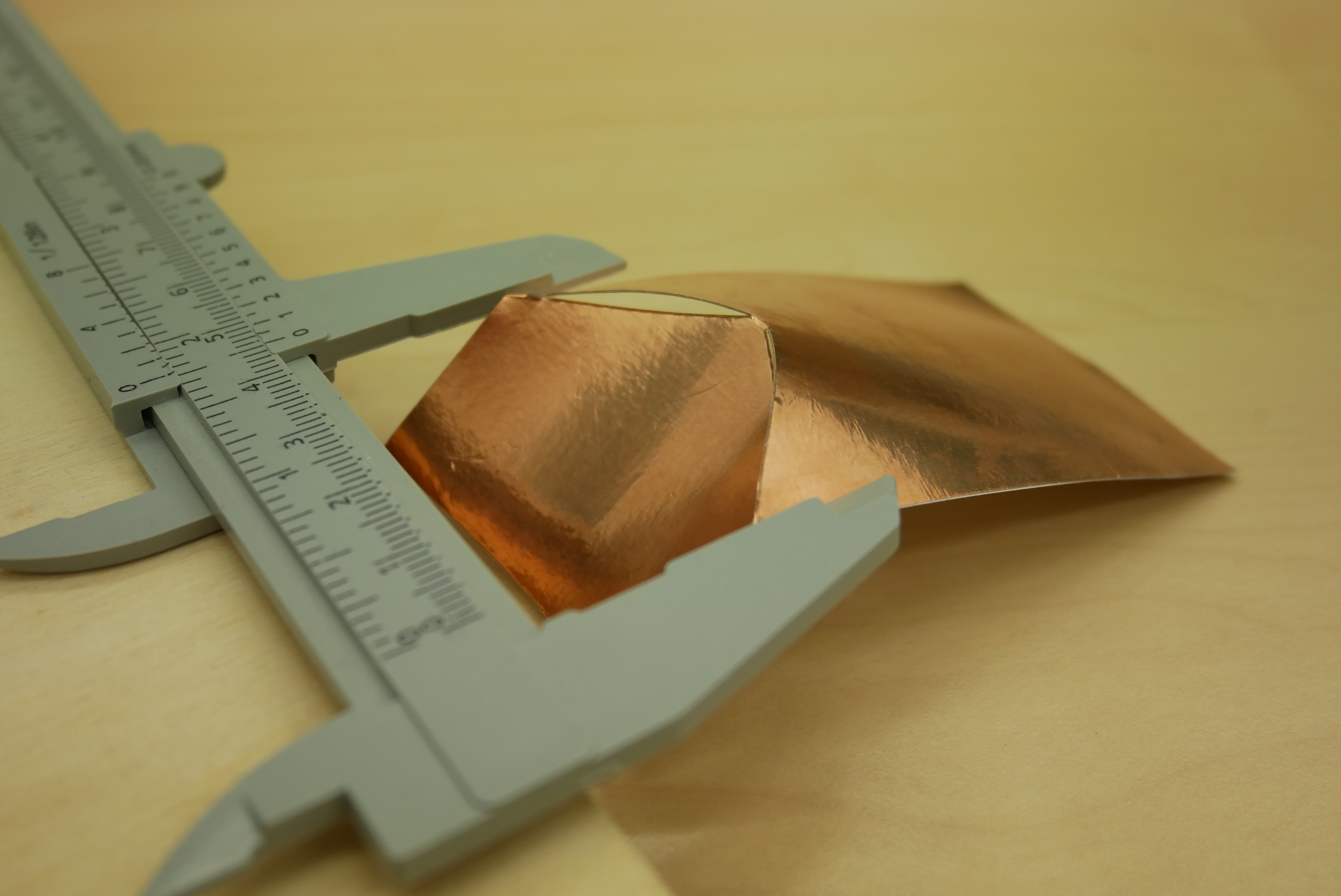} 
\caption{\label{fig:paradox_fold_real} Bending of an elastic plate via compressing 
the plate at the end-points of a crease line. Curved 
crease line (left), singularities occur when a polygonal approximation is used
(middle), these disappear if slits are introduced along the straight segments (right).}
\end{figure}

The real experiments can be simulated by means of a finite element discretization
of the bending-folding model. We use a mixed formulation to define a discrete second fundamental
form based
on the Hellan--Herrman--Johnson element. This element has the particular feature that
it can be used on curved elements, in contrast to many other finite element methods
developed for fourth order problems. \rx{Details of the numerical method are provided in 
Section~\ref{sec:num_exp}.} Figure~\ref{fig:paradox_fold_exp} shows the results
of canonical discretizations using curved crease-line approximations, piecewise straight
discrete crease lines, and straight crease lines imposing continuity only at the vertices. 
The deformations obtained with curved and slit approximations are nearly 
indistinguishable while the one with continuity and straight segments leads to 
a reduced deformation and stress concentrations at the vertices. 

\begin{figure}[htb]
\includegraphics[trim={0 0 0 1.695cm},clip,width=14cm]{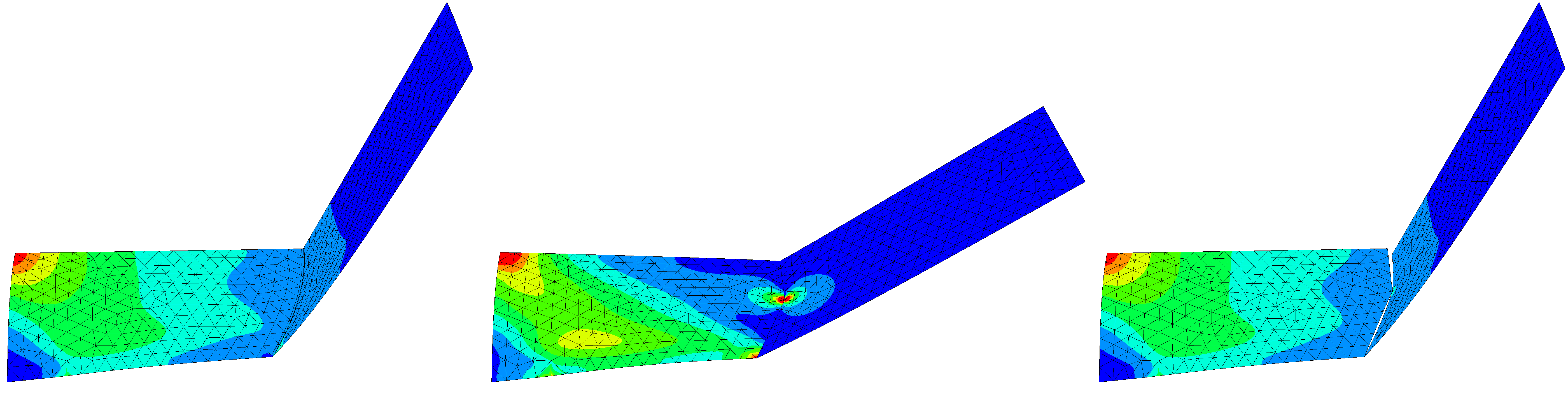}
\caption{\label{fig:paradox_fold_exp} Deformations and energy densities \rx{as coloring} 
in the simulation 
of folding and bending experiments using the symmetry of the problem along the long midline. 
A correct discrete deformation is obtained for a curved approximation of the crease line (left),
while the polygonal approximation leads to flatter pieces and a singularity (middle),
introducing discontinuities along the straight segments provides another correct 
approximation (right).}
\end{figure}

\subsection{Outline}
The article is organized as follows. A variational convergence result for approximations 
imposing continuity only in vertices is stated in Section~\ref{sec:gamma_conv}. In Section~\ref{sec:angle_curve_nonex} we recall an 
angle-curvature relation for folded isometries which implies that isometries are flat along 
straight segments of crease lines. This implies the nonexistence of nontrivially folded 
isometries for polygonal crease lines that are continuously differentiable in the subdomains. 
Additional results from numerical experiments are provided in Section~\ref{sec:num_exp}. 

\section{Discontinuous approximations}\label{sec:gamma_conv}
In this section we verify the convergence of the approximate minimization problems
which only impose continuity of deformations in the vertices of a polygonal crease line
approximation. To avoid assumptions about the extension of isometries we use a fattened 
crease line, cf. the right plot in Figure~\ref {fig:poly_crease}. This ensures that 
subdomains of approximating problems are subsets of the subdomains of the continuous
problem which allows for a straightforward identification of recovery sequences. 

Given a piecewise $C^1$ curve $\g$ parametrized by a function $b\in W^{1,\infty}(\a,\b;\R^2)$
with $|b'|=1$ which  partitions $\o$ into two disjoint Lipschitz domains
$\o_1$ and $\o_2$ we consider the minimization of the functional 
\[
I(v) = \frac12 \int_{\o\setminus \g} |D^2 v |^2 \dv{x},
\]
defined on the set of deformations that are isometric and piecewise $H^2$, i.e.,
on the set 
\[
V = \big\{v\in H^2(\o\setminus \g;\R^3)\cap W^{1,\infty}(\o;\R^3): 
(\nabla v)^\transp (\nabla v) = I_{2\times 2}\big\}.
\]
We note that $V\neq \emptyset$ holds since unfolded isometries are contained in $V$.
For a sequence of polygonal approximations $(\g_m)$ of $\g$ that are obtained by
linear interpolations $b_m$ of $b$ we choose matching, shape-regular
triangulations $(\cT_m)$ of $\o$ with vanishing maximal mesh-size as $m\to \infty$. 
We then consider
the fattened crease lines $\hg_m$ obtained as \rx{the union of triangles $T\subset \overline{\o}$
for which one side belongs to $\g_m$. We assume that the triangulation is sufficiently fine
so that the exact crease line $\g$ is contained in $\hg_m$,} cf.~Figure~\ref{fig:poly_crease}.
The fattened crease line gives rise to a disjoint partitioning (up to boundary points)
\[
\o = \ho_{m,1} \cup \hg_m \cup \ho_{m,2}
\]
such that $\ho_{m,\ell} \subset \o_\ell$ and $\g\subset \hg_m$. 
We thus consider the functionals
\[
\tI_m(v) = \frac12 \int_{\o\setminus \hg_m} |D^2 v |^2 \dv{x}
\]
defined on the set of isometric deformations that are continuous in the vertices $c_0,c_1,\dots,c_m$
of $\g_m$, and $H^2$-regular in the subdomains, i.e., on the set
\[\begin{split}
\tV_m &= \big\{v\in H^2(\o\setminus \hg_m;\R^3): 
(\nabla v)^\transp (\nabla v) = I_{2\times 2}, \\
&\qquad \qquad  v \text{ continuous in } c_0,c_1,\dots,c_m \big\}.
\end{split}\]
 We remark that we have $\tV_m\subset W^{1,\infty}(\o\setminus \hg_m;\R^3)$. 
We formally extend the functionals by $+\infty$ to deformations $v\in L^2(\o;\R^3)\setminus \tV_m$. 
Functions in $\tV_m$ and their derivatives are identified
with their trivial extensions to $\o$ throughout the following. 

A compactness result follows by an approximate extension of functions $v_m \in \tV_m$ to
functions $\tv_m \in W^{1,\infty}(\o;\R^3)$ via a nodal interpolation
of $v_m$ in the triangulations of $\o$ that contain the fattened crease lines $\hg_m$. 

\begin{proposition}[Compactness]\label{prop:compactness}
Let $(v_m) \subset L^2(\o;\R^3)$ be such that $v_m\in \tV_m$ and $\tI_m(v_m)\le c$
for all $m\in \N$. Then there exists a sequence $(\tv_m)\subset W^{1,\infty}(\o,\R^3)$
such that $\|\nabla \tv_m\|_{L^\infty(\o)} \le c'$ and $\tv_m-v_m \to 0$ in $H^1(\o';\R^3)$
for every $\o'$ that is compactly contained in $\o\setminus \g$. 
If $v\in W^{1,\infty}(\o;\R^3)$ is a weak-$\star$ accumulation point of $(\tv_m)$
then we have $v\in V$ and $D^2 v$ is a weak accumulation point in $L^2(\o;\R^{3\times 2 \times 2})$
of $(D^2v_m)$.
\end{proposition}

\begin{proof}
Let $(\cT_m)$ be a sequence of regular triangulations of $\o$ that match the
curves $(\g_m)$. Each $\cT_m$ contains the
triangles that define the fattened crease line $\hg_m$ such that the maximal
diameter $h_m$ tends to zero. Since all nodes of $\cT_m$ belong to 
the closure of $\o\setminus \hg_m$ the nodal interpolant $\cI_m v_m$ is
well defined for every $v_m \in \tV_m$. Moreover,  
$\|\nabla \cI_m v_m\|_{L^\infty(\o)} \le c_m \|\nabla v_m\|_{L^\infty(\o\setminus \hg_m)} \le c_m$. 
The uniform shape regularity implies that $c_m$ is uniformly bounded. 
Setting $\tv_m= \cI_m v_m$ thus proves the first part. 

Let $v \in W^{1,\infty}(\o;\R^3)$
be a weak-$\star$ accumulation point of $(\tv_m)$ and note that the (piecewise)
Hessians $D^2 v_m$ have a weak accumulation point $X$ in $L^2(\o;\R^{3\times 2\times 2})$
since $\tI_m(v_m)$ is bounded.
Since the differences $v_m-\tv_m$ converge strongly to zero in $H^1(\o';\R^3)$
for every set $\o'$ that is compactly contained in $\o\setminus \g$, it follows that 
$v\in H^2(\o\setminus \g;\R^3)$ with $D^2 v = X$ and 
$(\nabla v)^\transp (\nabla v) = I_{2\times 2}$, i.e., $v \in V$. 
\end{proof}

The convergence result is an immediate consequence of the construction of
the approximations. 

\begin{proposition}[Gamma convergence]\label{prop:convergence}
(i) If $(v_m)\subset L^2(\o;\R^3)$ is such that $v_m\in \tV_m$ and $\tI_m(v_m)\le c$ then 
there exists $v\in L^2(\o;\R^3)$ such that $v_m \to v$ in $L^2(\o;\R^3)$ (up to the selection of a subsequence)
and $v\in V$ with $I(v)\le \liminf_{m\to \infty} \tI_m(v_m)$.  \\
(ii) If $v\in V$ then there exists a sequence $(v_m)$ such that 
$v_m\in V_m$ for all $m$, $\lim_{m\to \infty} v_m = v$ in $L^2(\o;\R^3)$,  and 
$I(v) =\lim_{m\to \infty} I_m(v_m)$. 
\end{proposition}

\begin{proof}
(i) Using Proposition~\ref{prop:compactness} we obtain 
\rx{uniformly bounded} functions $\tv_m\in W^{1,\infty}(\o;\R^3)$.
After extraction of a subequence, this provides a limit 
$v\in W^{1,\infty}(\o;\R^3)$ with $\tv_m \wto^\star v$ in $W^{1,\infty}(\o;\R^3)$.
Proposition~\ref{prop:compactness} guarantees that $v\in V$ 
and $I(v) \le \liminf_{m\to\infty} I_m(v_m)$. 
This proves the first statement. \\
(ii) Given $v\in V$ we note that the restrictions $v_m = v|_{\ho_m}$
satisfy $v_m \in V_m$. Their trivial extensions converge in $L^2$ to $v$, and 
the energies converge, since $|\o_\ell \setminus \ho_{m,\ell}| \to 0$ as $m\to \infty$.
\end{proof}

The main implication of the convergence result concerns the accumulation of
almost-minimizers at minimizers.

\begin{corollary}[Convergence of almost-minimizers]
Assume that $(u_m)\subset L^2(\o;\R^3)$ is such that $u_m\in \tV_m$ and 
$\tI_m(u_m)\le \min_{v\in \tV_m} I(v) + \d_m$ for a sequence of positive numbers
$\d_m \to 0$. Then there exists a minimizer $u\in V$ for $I$ such that
$u_m\to u$ and $D^2 u_m \wto D^2 u$ weakly in $L^2(\o;\R^3)$. 
\end{corollary}

\section{Angle-curvature relations and nonexistence}\label{sec:angle_curve_nonex}

\subsection{Folding angle}
We consider a connected folding arc (segment) $\g$ that is parametrized by  
the embedded arclength curve  $b\in W^{2,\infty}(\a,\b;\R^2)$ and let 
$u\in H^2(\o\setminus \g;\R^3)$ be a folded isometry, i.e., 
\[
u \in V = 
\big\{ v \in H^2(\o\setminus \g;\R^3)\cap W^{1,\infty}(\o;\R^3):
 (\nabla v)^\transp (\nabla v) = I_{2\times 2}\big\}.
\]
The curve $\g$ is assumed 
to partition $\o$ into two subdomains $\o_1$ and $\o_2$. We further assume that
the restrictions $u_\ell = u|_{\o_\ell}$, $\ell=1,2$, can be extended as  
$H^2$ isometries to open neighborhoods of $\overline{\o}_\ell$.  
The mapping $u\circ b$ provides an arclength parametrization of the 
deformed folding arc with unit tangent vector 
\[
t = \g' = (Du\circ b) b',
\] 
which is the same for both $u_\ell$, $\ell=1,2$, cf. Figure~\ref{fig:folded_isos}.

\begin{figure}
\scalebox{1.1}{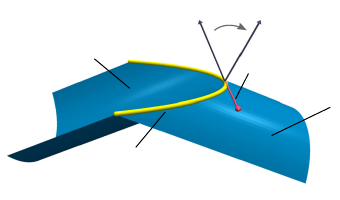}
\caption{\label{fig:folded_isos}
Folding angle between the jumping normals along a crease 
which partitions the deformed sheet into  parts of opposite curvatures  and defines
two isometries of the subdomains. The induced Darboux 
frames specify curvature and torsion quantities for the deformed folding arc~$u(\g)$.} 
\end{figure}

We let $n_\ell = \p_1u_\ell \times \p_2u_\ell: \overline{\o}_\ell \to \R^3$, $\ell=1,2$, 
denote the continuous
unit normals to the deformed adjacent surfaces in a neighborhood of $\g$. Along $\g$
we identify these and other quantities with mappings defined on the interval $[\a,\b]$ via,
e.g., $n_\ell(s) = n_\ell(b(s))$. By the extensibility assumption 
and the regularity result~\cite{Kirc01,MulPak05} we have that $n_\ell \in C([\a,\b];\R^3)$  
and 
\begin{equation}\label{eq:rotate_normal}
n_2 = R(\theta,t) n_1,
\end{equation}
where $R(\theta,t)$ is the rotation about $t$ by the angle $\theta$, 
which satisfies
\begin{equation}\label{eq:fold_angle}
\cos \theta = n_1 \cdot n_2.
\end{equation}
Choosing conormal vectors $m_\ell = n_\ell \times t$, $\ell=1,2$, that
are tangential to the surfaces and normal to the folding curve, we consider
the Darboux frames
\[
r_\ell = [t,m_\ell,n_\ell],
\]
for $\ell=1,2$. The second and third column vectors of the frames are related via
\[
\begin{split}
m_2 &= \cos (\theta) m_1 + \sin (\theta) n_1, \\
n_2 &= -\sin (\theta) m_1 + \cos (\theta) n_1.
\end{split}
\]
The frames give rise to the geodesic and normal 
curvatures
\[
\k_\ell = t'\cdot m_\ell, \quad \mu_\ell = t'\cdot n_\ell,
\]
and the geodesic torsions
\[
\tau_\ell = (m_\ell)' \cdot n_\ell.
\]
Since the isometric deformations $u_\ell$ preserve intrinsic quantities
we have that 
\[
\k = \k_1 = \k_2.
\]
The combination of this identity with the equations for $m_2$ and $n_2$ 
implies that we have
\[
\k  = t' \cdot m_2 = t' \cdot (\cos (\theta) m_1 + \sin (\theta) n_1) 
= \cos (\theta) \k + \sin (\theta) \mu_1,
\]
i.e., 
\[
(1-\cos (\theta)) \k = \sin (\theta)  \mu_1.
\]
Analogously, using that $m_1 = \cos (\theta) m_2 - \sin (\theta) n_2$,
we find that
\[
(1-\cos (\theta)) \k = - \sin (\theta)  \mu_2.
\] 
Incorporating the trigonometric identities $\cos (2\a) = \cos^2 (\a) - \sin^2 (\a)$ and 
$\sin(2\a) = 2 \cos (\a) \sin (\a)$ we deduce that
\[
2 \sin^2\Big(\frac\theta2\Big) \k 
= \pm  2 \sin\Big(\frac\theta2\Big) \cos \Big(\frac\theta2\Big) \mu_\ell.
\]
For the geodesic torsions we similarly derive, using the orthogonality of the column
vectors of $r_\ell$, that 
\[\begin{split}
\tau_2 &= (m_2)' \cdot n_2 = 
(\cos (\theta) m_1 + \sin (\theta) n_1)' \cdot (-\sin (\theta) m_1 + \cos (\theta) n_1) \\
&= \theta'  (-\sin (\theta) m_1 + \cos (\theta) n_1) \cdot (-\sin (\theta) m_1 + \cos (\theta) n_1) \\
& \quad + \Big( \cos (\theta) (m_1)' + \sin (\theta) (n_1)'\Big) \cdot (-\sin (\theta) m_1 + \cos (\theta) n_1) \\
&= \theta' (\sin^2 (\theta)  + \cos^2 (\theta)) 
+ \cos^2 (\theta) (m_1)' \cdot n_1 - \sin^2 (\theta) (n_1)' \cdot m_1 \\
&= \theta' + \tau_1.
\end{split}\]

The calculations imply the following result from~\cite{Horn23-pre}, which extends
observations from~\cite{DunDun82}.

\begin{proposition}[Folding angle, \cite{Horn23-pre}]\label{prop:folding_angle}
Let $u\in V$ and let $b\in W^{2,\infty}(\a,\b;\R^2)$ 
be an arclength parametrization for $\g$. Assume that the restrictions of
$u$ to the subdomains $\o_\ell$ can be extended as $H^2$ isometries
to open neighborhoods of $\overline{\o}_\ell$ for $\ell=1,2$.
There exists a well defined {\em folding angle} $\theta \in C([\a,\b])$ 
satisfying~\eqref{eq:rotate_normal} such that 
\begin{equation}\label{eq:curv_angle}
\k \sin \Big(\frac\theta2\Big) = \cmu \cos\Big(\frac\theta2\Big),
\end{equation}
with $\cmu = \mu_1$,  wherever the surface is folded, i.e., 
$\theta \not \in 2\pi \Z$. The induced normal curvatures and 
geodesic torsions are related via $\mu_2 = - \mu_1$ wherever
$\theta  \not \in 2\pi \Z$, and $\tau_2 = \tau_1 + \theta'$.
\end{proposition}

The identities imply that, e.g., if $\k=0$ then $\theta \in \pi \Z$, i.e., $u$ is unfolded
or folded back, or $\cmu =0$. In the latter case it can be shown that the folding angle $\theta$
is constant, cf.~\cite{Horn11} for details. Whenever $\k\neq 0$ we have that
the folding angle is zero or uniquely defined via $\theta = 2 \arctan(\cmu/\k)$.

\subsection{Failure of convergence}
The Babu\v{s}ka paradox arises in the context of the nonlinear bending-folding
model via a nonexistence result of folded $H^2$-isometries for polygonal crease lines. 

\begin{proposition}[Nonexistence]\label{prop:nonexistence-b}
Let $\g_a$, $\g_b\subset\R^2$ be nondegenerate
disjoint open line segments with a common endpoint $x_C$. Let $\o\subset\R^2$ be simply connected
and such that $\g = \g_a\cup\g_b\cup \{x_C\}$ is contained
in $\o$ and both endpoints of $\g$ belong to $\p\o$, cf.~Figure~\ref{fig:nonexistence}.
Denote the connected components of $\o\setminus\g$ by $\o_1$ and $\o_2$.
Let $u_\ell \in C^1\left(\overline{\o}_\ell; \R^3\right)$ and assume that
$u_1 = u_2$ on $\g$. \rx{Then, we have that:}  \\
(i) If $\g_a,\g_b$ are not parallel, then $D u_1(x_C) = D u_2(x_C)$. \\
(ii) If, furthermore, \rx{each $u_\ell$ is the restriction of an $H^2$ 
isometric immersion $\widetilde{u}_\ell$ defined on an open neighbourhood of $\gamma_a\cup\gamma_b$}, 
then $D u_1 = Du_2$ along $\g$.
\end{proposition}

\begin{proof}
 (i) Since the tangential derivatives of $u_1$ and $u_2$ coincide along the segments
$\g_a$ and $\g_b$ we have that $Du_1(x_C) \tau_a = Du_2(x_C) \tau_a$ as well as 
$Du_1(x_C) \tau_b = Du_2(x_C) \tau_b$ with the linearly independent 
tangent vectors $\tau_a,\tau_b\in \R^2$ so that $Du_1(x_C)=Du_2(x_C)$. \\
(ii) \rx{Applying \cite[Corollary 2.6]{Horn23-pre} to the Darboux frames associated
with the immersions $\widetilde{u}_\ell$ implies that the folding angle is constant on
$\gamma_a$. The same is true on $\gamma_b$.
Hence, by continuity of the normals along $\g$, the resulting continuity of $\theta$ 
by~\eqref{eq:fold_angle}, and~(i),
the folding angle must be zero on all of $\g$. This implies that the normals
$n_1$ and $n_2$ coincide along $\g$ and hence that the deformation gradients
are identical on $\g$.}
\end{proof}

\begin{figure}[htb]
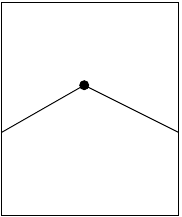
\caption{\label{fig:nonexistence} Polygonal crease $\g=\g_1\cup \g_2$ line with vertex $x_C$}
\end{figure}

\begin{remark}
The proposition implies that also accumulation points of piecewise $C^1$ regular isometric
deformations in $V(\g_m)$ are unfolded for a sequence of polygonal crease lines. If curved segments
are used then the folding angle vanishes at the vertices but may be different from zero between 
vertices, and the numerical experiment illustrated in the left plot of \rx{Figure~\ref{fig:paradox_fold_exp}} 
indicates correct convergence. 
\end{remark}

\section{Numerical experiments}\label{sec:num_exp}
We describe in this section the numerical method that leads to the approximations shown in
Figure~\ref{fig:paradox_fold_exp}. For related numerical methods to approximate isometric deformations 
we refer the reader to~\cite{Bart15-book,RuSiSm22,BoGuMo23,BGNY23}.

\subsection{Experimental setting}
We consider the rectangular domain $\o =(0,2)\times (-1/2,1/2)$ and define a crease line
as the \rx{intersection} of $\overline{\o}$ with a circle of unit radius, i.e., 
$\g = \overline{\o}\cap \p B_1(0)$. The devised numerical methods impose \rx{different} continuity
conditions along $\g$ and its approximations. \rx{Furthermore, we impose that no bending moments 
are transferred while no continuity is imposed on the deformation gradient.} We impose compressive boundary
conditions along \rx{the opposite boundary parts} $\g_D = [0,0.2]\times\{\pm 0.5\}$ \rx{in one of the subdomains,
cf.~Figure~\ref{fig:problem_geometry}.} \rx{The boundary conditions are continuously increased} 
via a pseudo time $t\in[0,1]$, i.e., we set 
\[
u_D(t,x,y) = \big[x,y,0\big]^\transp + \big[0, -(1/10)y(1-x) t, 0 \big]^\transp.
\]
To guarantee that the sheet bends upwards we include a uniform vertical force that
is zero at $t=1$, i.e., 
\[
f(t) = \big[0,0,(2/5)(1-t^2)\big]^\transp.
\]
We reduce the computational effort by exploiting the symmetry of the setting and 
only discretize the subdomain $\o' =(0,2)\times (-1/2,0)$ imposing appropriate boundary
conditions along the symmetry axis $\g_\sym = [0,2] \times \{0\}$. We use the Young's 
modulus $E=10$ and a vanishing
Poisson ratio which is compatible with the isometry condition. We consider three
different treatments of the crease line and the continuity condition:
\begin{itemize}
\item[(S1)] The crease line is isoparametrically resolved by the numerical method and continuity is 
imposed along the entire crease line. 
\item[(S2)] The crease line is approximated by a polygonal curve $\g_\ell$ on which continuity
is imposed. 
\item[(S3)] The crease is approximated by a polygonal curve $\g_\ell$ and continuity
is imposed at the vertices of $\g_\ell$.
\end{itemize}
\rx{These settings} are illustrated in Figure~\ref{fig:problem_geometry}. In view of our theoretical
results we expect the formation of certain singularities at the vertices of the polygonal
crease line in setting~(S2) and correct approximations in settings~(S1) and~(S3).
\rx{These effects are confirmed by the results of simulations shown in Figure~\ref{fig:paradox_fold_exp}.
A zoom towards the simulation results in neighborhoods of the different discrete crease lines 
is shown in Figure~\ref{fig:paradox_fold_exp_zoom}, which also shows the effect of the
geometric refinement in case of the crease line approximation (S1).}

\begin{figure}[ht!] 
  \scalebox{1.1}{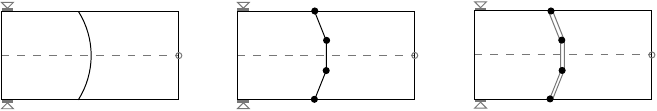}
  \caption{Experimental setting with boundary conditions (arrows), symmetry
axis (dashed), boundary point $x_P$ (gray dot), and crease line treatments (S1)-(S3) from
left to right.}
  \label{fig:problem_geometry}
\end{figure}

\begin{figure}[htb]
\includegraphics[trim={0 0 0 1.695cm},clip,width=14cm]{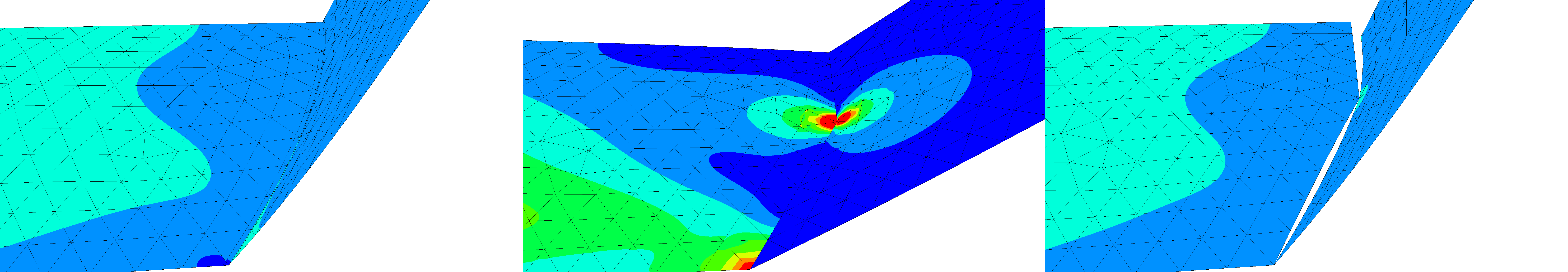}
\caption{\label{fig:paradox_fold_exp_zoom} \rx{Zoom towards discrete crease lines for the 
folding and bending experiments corresponding to the full deformations shown 
in Figure~\ref{fig:paradox_fold_exp} with crease line approximations (S1), (S2), and (S3)
(from left to right) and geometric refinements towards the entire crease line in case of (S1) 
and towards the corner points in case of (S2) and (S3).}}
\end{figure}

\subsection{Saddle-point formulation}
We reformulate the variational formulation of the bending-folding model as a
saddle-point problem in order to use the methods devised for Koiter shells in 
\cite{NS19,NS24}. We note that for isometric deformations we have the relation
$|D^2 v|= |\cH_n|$ for the corresponding Frobenius norms and the second fundamental
form given by 
\[
\cH_n = -(\nabla v)^\transp \nabla \rx{n}, \quad 
\rx{n = \frac{\partial_xv\times \partial_yv}{|\partial_xv\times \partial_yv|},}
\]
\rx{where $n$ is a unit normal vector on the deformed surface.} 
We incorporate the second fundamental form in terms of a mixed formulation by 
introducing the bending moment tensor as the energetic conjugate to the curvature
tensor, i.e.,  
\[
\bM = -\frac{E}{12} (\nabla v)^\transp \nabla n.
\]
The resulting Hellinger--Reissner two-field formulation includes the isometry 
condition via a penalty term and is given by the Lagrange functional 
\[\begin{split}
\L(v,\bM) = \int_\o& \alpha|(\nabla v)^\top(\nabla v)-I_{2\times 2}|^2 
 -\frac{6}{E}|\bM|^2 + ((\nabla v)^\transp \nabla n):\bM -f\cdot v \dv{x}.
\end{split}\]
We thus aim at approximating a saddle-point $\min_v \max_{\bM} \L(v,\bM)$ 
imposing the boundary conditions 
\[
v=u_D \text{ on } \g_D, \qquad \bM_{\nu\nu}=0 \text{ on } \p\o, \qquad v_y=0 \text{ on } \g_{\sym},
\]
where $\bM_{\nu\nu} = (\bM \nu)\cdot \nu$, 
and the different continuity conditions on the crease line $\g$ of settings~(S1)-(S3) together
with the condition $\bM_{\nu\nu}=0$ on $\g$ or $\g_\ell$ for the traces from both sides with 
a unit normal $\nu$ along $\g$.

\subsection{Hellan--Herrmann--Johnson method}
We let $\cT_h$ be a triangulation of $\o$, which contains polynomially curved elements
in setting~(S1). On interelement boundaries we consider the clockwise oriented unit tangent
vector $t$ and let $\nu$ be an outward unit normal.
The set of edges of elements is denoted by $\cE_h$. The jump of an elementwise continuous
quantity $w$ over an inner edge $E  = T_+ \cap T_- \in \cE_h$ with a fixed unit normal 
pointing from $T_-$ into $T_+$ is defined as 
\[
\jump{w}|_E = w|_{T_+} - w|_{T_-},
\]
for boundary edges we simply set $\jump{w}|_E = w|_E$. 
The deformation and bending moment fields are discretized 
by Lagrangian and Hellan--Herrmann--Johnson \cite{Hel67,Her67,Joh73,Com89,Walk22} finite elements, respectively,
\[\begin{split}
  &\mathcal{U}^k= \{ v\in C^0(\omega,\R^3)\,:\, v|_{T}\in \tcP_k(T,\R^3)\,\text{for all }  T\in\cT_h\},\\
  &\mathcal{M}^k=\{ \bM\in L^2(\omega,\R^{2\times 2}_{\sym})\,:\, \bM|_T\in \tcP_k(T, \R^{2\times 2}_{\sym}) \,\text{for all } T\in\cT_h \text{ and } \\
  & \qquad \qquad \llbracket \bM_{\nu\nu} \rrbracket_E=0\,\text{for all } E\in \E_h\},
\end{split}\]
where $\tcP_k(T)$ denotes the set of functions that are
(isoparametrically) transformed polynomials. \rx{The Hellan--Herrman--Johnson stress elements 
satisfy the continuity of the normal-normal component of the bending stress by construction. 
For the normal-tangential stress component, and thus the physical normal continuity of the stress, 
this follows in weak sense by the discretization method. This treatment allows for a simple
 construction of stress elements that are requrired to be both symmetric and normal continuous.}
We use the strategies developed in~\cite{NS19,NSS23,NS24} 
to obtain an approximation of the Lagrange functional that can be applied to deformations
that are merely continuous. Despite the elementwise second fundamental form $\nabla_hn$, 
which is discontinuous over sides of elements, the angle $\arccos(n^+\cdot n^-)$ of the 
jump of the normal vectors across elements is considered at the edges.
For a discrete pair $(v,\bM)\in \mathcal{U}^k\times \mathcal{M}^{k-1}$ the Lagrange functional
is defined via 
\[\begin{split}
\L^{\mathrm{HHJ}}(v,\bM) 
= &  \int_{\cup \cT_h} \alpha|(\nabla v)^\transp (\nabla v )-I_{2\times 2}|^2 -\frac{6}{E}|\bM|^2- \Hessian_n:\bM\dv{x} \\
& - \int_\o f \cdot v \dv{x} + \int_{\cup \cE_h} \arccos(n^+ \cdot n^-)\,\bM _{\nu\nu}\dv{s}, 
\end{split}\]
with the elementwise defined second fundamental form $\Hessian_n= \sum_{i=1}^3\nabla_h^2 v_i\,n_i$. 
The boundary conditions on $\g_D$ are also approximated using a penalty term. This allows us to adaptively increase 
the penalty parameters to enforce the boundary condition and isometry constraint in the 
first load-step. This additional control over the constraints improves the convergence of
the employed Newton's method. For further implementation  details we refer to \cite{NS19,NS24}.
We remark that using the norm of the Hessian as in~\cite{Wal2024} instead of the second fundamental form introduces
a simpler structure but leads to additional equations. Corresponding experimental results were nearly identical. 

\subsection{Numerical results}
We use uniform mesh refinements and additional local geometric refinements with refinement factor $0.125$ 
towards $\g$ in setting (S1) and the interior vertices in case of (S2) and (S3). The reported
results correspond to cubic polynomials for the deformation, i.e., we always set $k=3$, and a fixed
polygonal crease line $\g_\ell$ with four vertices in the full domain $\overline{\o}$.
To compare the experimental results for the different approximations of the crease line,
we plotted in Figure~\ref{fig:z_disp} the vertical deflection $u_h\cdot e_3$ at the boundary
point \rx{$x_p = (2,0)$.} We observe a reduced deflection in setting~(S2) in comparison
to settings~(S1) and~(S3), confirming the expected locking effect if continuity is imposed 
along a polygonal crease line. The failure of incorrect convergence is also visible in the
different norms of the bending moment $\bM$ shown in Tables~\ref{tab:res1} and~\ref{tab:res2}.
While the results for settings~(S1) and~(S3) nearly coincide, an incorrect $L^2$ norm and
a rapidly growing $L^{64}$ norm, \rx{which we use to approximate the $L^\infty$ norm}, are obtained. 
The singular effect in setting (S2) becomes worse when 
a second geometric refinement is carried out, i.e., we obtained 
$\|\bM\|_{L^{64}} = 923.670,\, 1644.595,\, 2430.662$ for $h=0.2,\,0.1,\,0.05$, respectively. 

In the experiments we used the (final) penalty parameter $\a=10^6$ for both, the constraint 
approximation and enforcement of the compressive boundary condition. We used 30 uniform 
load-steps for the pseudo time $t\in [0,1]$. To improve the convergence behavior of 
Newton's method we use a damped version with parameter $\eta=0.05$. Further, for the first 
load-step we used an internal loop to gradually reduce the violation of the  
isometry and deformation boundary constraints by the function $\sqrt{\beta}$ with $\beta\in[0,1]$ 
to ensure that the first intermediate configuration is reached. We used 5 uniform internal 
steps for this purpose. We stopped Newton's method when a residual of $10^{-6}$ was reached.

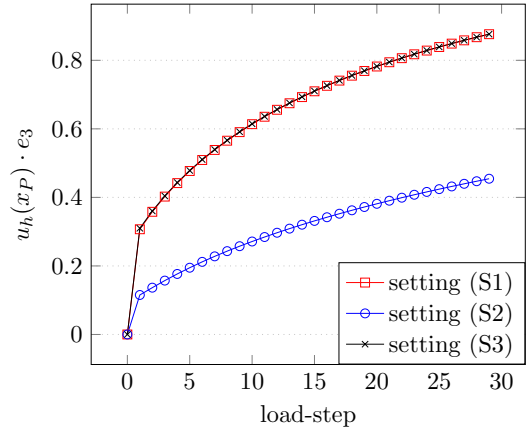
\begin{figure}[ht!]
  \centering
  \resizebox{0.55\textwidth}{!}{\begin{tikzpicture}
		\begin{axis}[
			legend style={at={(1,0)}, anchor=south east},
			xlabel={load-step},
			ylabel={$u_h(x_P)\cdot e_3$},
			ymajorgrids=true,
			grid style=dotted,
			]

      \addlegendentry{setting (S1)}
      \addplot[color=red, mark=square, style=solid]
      coordinates {
        (0,0)
        (1,0.30704)
        (2,0.35795)
        (3,0.40233)
        (4,0.4417)
        (5,0.47705)
        (6,0.50913)
        (7,0.53846)
        (8,0.56546)
        (9,0.59045)
        (10,0.61368)
        (11,0.63535)
        (12,0.65565)
        (13,0.67471)
        (14,0.69266)
        (15,0.7096)
        (16,0.72562)
        (17,0.7408)
        (18,0.7552)
        (19,0.7689)
        (20,0.78194)
        (21,0.79437)
        (22,0.80623)
        (23,0.81757)
        (24,0.82841)
        (25,0.83879)
        (26,0.84873)
        (27,0.85827)
        (28,0.86743)
        (29,0.87622)
      };

      \addlegendentry{setting (S2)}
      \addplot[color=blue, mark=o, style=solid]
      coordinates {
        (0,0)
        (1,0.11522)
        (2,0.13658)
        (3,0.15719)
        (4,0.1766)
        (5,0.19479)
        (6,0.21188)
        (7,0.22796)
        (8,0.24316)
        (9,0.25757)
        (10,0.27128)
        (11,0.28435)
        (12,0.29685)
        (13,0.30883)
        (14,0.32034)
        (15,0.33141)
        (16,0.34208)
        (17,0.35238)
        (18,0.36233)
        (19,0.37197)
        (20,0.3813)
        (21,0.39035)
        (22,0.39914)
        (23,0.40769)
        (24,0.416)
        (25,0.42409)
        (26,0.43197)
        (27,0.43966)
        (28,0.44715)
        (29,0.45447)
      };

      \addlegendentry{setting (S3)}
    \addplot[color=black, mark=x, style=solid]
    coordinates {
      (0,0)
      (1,0.3099)
      (2,0.36022)
      (3,0.40422)
      (4,0.44333)
      (5,0.47851)
      (6,0.51046)
      (7,0.5397)
      (8,0.56662)
      (9,0.59155)
      (10,0.61472)
      (11,0.63634)
      (12,0.65659)
      (13,0.6756)
      (14,0.69349)
      (15,0.71038)
      (16,0.72635)
      (17,0.74148)
      (18,0.75583)
      (19,0.76948)
      (20,0.78246)
      (21,0.79484)
      (22,0.80665)
      (23,0.81793)
      (24,0.82872)
      (25,0.83905)
      (26,0.84894)
      (27,0.85843)
      (28,0.86753)
      (29,0.87627)
    };
		
		\end{axis}
\end{tikzpicture}}

\caption{Evolution of vertical deflections during the pseudo time-stepping 
at \rx{$x_p=(2,0)$} for crease line approximations of settings 
(S1)-(S3) and triangulations with mesh size $h=0.05$.}
\label{fig:z_disp}
\end{figure}

\begin{table}[h!]
  \begin{tabular}{c|cc||cc||cc}
    $h$ & $\|\bM\|_{L^2}$ & $\|\bM\|_{L^{64}}$ & $\|\bM\|_{L^2}$ & $\|\bM\|_{L^{64}}$ & $\|\bM\|_{L^2}$ & $\|\bM\|_{L^{64}}$\\  \hline
    0.2  & 6.841 & 22.173 &  7.417 & 27.226 & 6.838 & 21.938\\
    0.1  & 6.832 & 22.042 &  7.457 & 50.182 & 6.827 & 21.704\\
    0.05 & 6.832 & 22.512 &  7.503 & 95.972 & 6.828 & 21.714 \\[2mm]
  \end{tabular} 
  \caption{$L^p$ norms of the bending moment tensor for uniform mesh refinement in settings (S1)-(S3) from left to right.}
  \label{tab:res1}
\end{table}

\begin{table}[h!]
  \begin{tabular}{c|cc||cc||cc}
    $h$ &  $\|\bM\|_{L^2}$  & $\|\bM\|_{L^{64}}$&  $\|\bM\|_{L^2}$  & $\|\bM\|_{L^{64}}$&  $\|\bM\|_{L^2}$  & $\|\bM\|_{L^{64}}$ \\    \hline
    0.2  & 6.841  & 22.188 & 7.604  & 161.623 & 6.838  & 21.936 \\ 
    0.1  & 6.832  & 22.043 & 7.582  & 307.049 & 6.827  & 21.703 \\  
    0.05 & 6.832  & 22.070 & 7.614  & 589.891 & 6.828  & 21.714 \\[2mm] 
  \end{tabular}
  \caption{$L^p$ norms of the bending moment tensor for uniform mesh refinement with one geometric refinement in settings (S1)-(S3) from left to right.}
  \label{tab:res2}
\end{table}


\subsection*{Acknowledgments} 
The authors SB and PH acknowledge support by the DFG via the priority programme
SPP 2256 {\em Variational Methods for Predicting Complex Phenomena in Engineering 
Structures and Materials} (441528968 (BA 2268/7-2, HO 4697/2-2).  
The author AB is partially supported by NSF grant DMS-2409807.


\section*{References}
\printbibliography[heading=none]

\end{document}

%% file: sb_macros.tex
\def\R{\mathbb{R}}

\def\N{\mathbb{N}}

\def\Z{\mathbb{Z}}

\def\cE{\mathcal{E}}

\def\cH{\mathcal{H}}
\def\cI{\mathcal{I}}

\def\cP{\mathcal{P}}

\def\cT{\mathcal{T}}

\def\a{\alpha}
\def\b{\beta}
\def\g{\gamma}
\def\d{\delta}

\def\k{\kappa}

\def\s{\sigma}
\def\p{\partial}
\def\o{\omega}

\def\G{\Gamma}

\def\wto{\rightharpoonup}
\def\transp{{\sf T}}

\newcommand{\dv}[1]{\,{\mathrm d}#1}

\newcommand{\wcheck}[1]{#1\hspace{-.8ex}\mbox{\huge {\lower.45ex \hbox{$\textstyle \check{}$}}} \hspace{.5ex}}

\newcommand{\jump}[1]{\llbracket#1\rrbracket}   

\let\oldmarginpar\marginpar
\renewcommand\marginpar[1]{
  \oldmarginpar[\raggedleft\footnotesize #1]
  {\raggedright\footnotesize #1}}

\newtheorem{definition}{Definition}

\newtheorem{proposition}[definition]{Proposition}

\newtheorem{corollary}[definition]{Corollary}
\newtheorem{remark}[definition]{Remark}

\numberwithin{definition}{section}
\definecolor{modmag}{RGB}{179,0,229}

%% file: figs/poly_crease.pdf_t
\begin{picture}(0,0)%
\includegraphics{poly_crease.pdf}%
\end{picture}%
\setlength{\unitlength}{4144sp}%
\begingroup\makeatletter\ifx\SetFigFont\undefined%
\gdef\SetFigFont#1#2#3#4#5{%
  \reset@font\fontsize{#1}{#2pt}%
  \fontfamily{#3}\fontseries{#4}\fontshape{#5}%
  \selectfont}%
\fi\endgroup%
\begin{picture}(4993,1824)(1339,-1423)
\put(2026, 29){\makebox(0,0)[b]{\smash{{\SetFigFont{10}{12.0}{\familydefault}{\mddefault}{\updefault}{\color[rgb]{0,0,0}$\o^1$}%
}}}}
\put(3826, 29){\makebox(0,0)[b]{\smash{{\SetFigFont{10}{12.0}{\familydefault}{\mddefault}{\updefault}{\color[rgb]{0,0,0}$\o_m^1$}%
}}}}
\put(3466,-601){\makebox(0,0)[lb]{\smash{{\SetFigFont{10}{12.0}{\familydefault}{\mddefault}{\updefault}{\color[rgb]{0,0,0}$\g_m$}%
}}}}
\put(5626, 29){\makebox(0,0)[b]{\smash{{\SetFigFont{10}{12.0}{\familydefault}{\mddefault}{\updefault}{\color[rgb]{0,0,0}$\ho_m^1$}%
}}}}
\put(2026,-1141){\makebox(0,0)[b]{\smash{{\SetFigFont{10}{12.0}{\familydefault}{\mddefault}{\updefault}{\color[rgb]{0,0,0}$\o^2$}%
}}}}
\put(3826,-1141){\makebox(0,0)[b]{\smash{{\SetFigFont{10}{12.0}{\familydefault}{\mddefault}{\updefault}{\color[rgb]{0,0,0}$\o_m^2$}%
}}}}
\put(5626,-1141){\makebox(0,0)[b]{\smash{{\SetFigFont{10}{12.0}{\familydefault}{\mddefault}{\updefault}{\color[rgb]{0,0,0}$\ho_m^2$}%
}}}}
\put(1576,-601){\makebox(0,0)[lb]{\smash{{\SetFigFont{10}{12.0}{\familydefault}{\mddefault}{\updefault}{\color[rgb]{0,0,0}$\g$}%
}}}}
\put(5446,-691){\makebox(0,0)[lb]{\smash{{\SetFigFont{10}{12.0}{\familydefault}{\mddefault}{\updefault}{\color[rgb]{0,0,0}$\hg_m$}%
}}}}
\end{picture}%

%% file: figs/frames_sketch.pdf_t
\begin{picture}(0,0)%
\includegraphics{frames_sketch.pdf}%
\end{picture}%
\setlength{\unitlength}{4144sp}%
\begin{picture}(2580,1515)(1126,-3026)
\put(3691,-2356){\makebox(0,0)[lb]{\smash{\fontsize{10}{12}\normalfont {\color[rgb]{0,0,0}$u(\o_2)$}%
}}}
\put(1801,-1996){\makebox(0,0)[rb]{\smash{\fontsize{10}{12}\normalfont {\color[rgb]{0,0,0}$u(\o_1)$}%
}}}
\put(2611,-1771){\makebox(0,0)[rb]{\smash{\fontsize{10}{12}\normalfont {\color[rgb]{0,0,0}$n_1$}%
}}}
\put(3151,-1861){\makebox(0,0)[lb]{\smash{\fontsize{10}{12}\normalfont {\color[rgb]{0,0,0}$n_2$}%
}}}
\put(2881,-1636){\makebox(0,0)[b]{\smash{\fontsize{10}{12}\normalfont {\color[rgb]{0,0,0}$\theta$}%
}}}
\put(3050,-2112){\makebox(0,0)[lb]{\smash{\fontsize{10}{12}\normalfont {\color[rgb]{0,0,0}$t$}%
}}}
\put(1981,-2806){\makebox(0,0)[lb]{\smash{\fontsize{10}{12}\normalfont {\color[rgb]{0,0,0}$u(\g)$}%
}}}
\end{picture}%

%% file: figs/nonexistence.pdf_t
\begin{picture}(0,0)%
\includegraphics{nonexistence.pdf}%
\end{picture}%
\setlength{\unitlength}{4144sp}%
\begingroup\makeatletter\ifx\SetFigFont\undefined%
\gdef\SetFigFont#1#2#3#4#5{%
  \reset@font\fontsize{#1}{#2pt}%
  \fontfamily{#3}\fontseries{#4}\fontshape{#5}%
  \selectfont}%
\fi\endgroup%
\begin{picture}(1374,1644)(1339,-1333)
\put(2026,-286){\makebox(0,0)[lb]{\smash{{\SetFigFont{10}{12.0}{\familydefault}{\mddefault}{\updefault}{\color[rgb]{0,0,0}$x_C$}%
}}}}
\put(2026, 29){\makebox(0,0)[b]{\smash{{\SetFigFont{10}{12.0}{\familydefault}{\mddefault}{\updefault}{\color[rgb]{0,0,0}$\o_1$}%
}}}}
\put(2026,-1141){\makebox(0,0)[b]{\smash{{\SetFigFont{10}{12.0}{\familydefault}{\mddefault}{\updefault}{\color[rgb]{0,0,0}$\o_2$}%
}}}}
\put(1621,-466){\makebox(0,0)[rb]{\smash{{\SetFigFont{10}{12.0}{\familydefault}{\mddefault}{\updefault}{\color[rgb]{0,0,0}$\g_a$}%
}}}}
\put(2341,-466){\makebox(0,0)[lb]{\smash{{\SetFigFont{10}{12.0}{\familydefault}{\mddefault}{\updefault}{\color[rgb]{0,0,0}$\g_b$}%
}}}}
\end{picture}%

%% file: figs/settings.pdf_t
\begin{picture}(0,0)%
\includegraphics{settings.pdf}%
\end{picture}%
\setlength{\unitlength}{4144sp}%
\begin{picture}(4999,839)(1339,-2391)
\put(2116,-1861){\makebox(0,0)[lb]{\smash{\fontsize{9}{10.8}\normalfont {\color[rgb]{0,0,0}$\g$}%
}}}
\put(3916,-1861){\makebox(0,0)[lb]{\smash{\fontsize{9}{10.8}\normalfont {\color[rgb]{0,0,0}$\g_m$}%
}}}
\put(5716,-1861){\makebox(0,0)[lb]{\smash{\fontsize{9}{10.8}\normalfont {\color[rgb]{0,0,0}$\g_m^{\text{slit}}$}%
}}}
\end{picture}%